\documentclass[12pt,twoside]{amsart}
\usepackage{amssymb}
\usepackage{amscd}
\usepackage{amsmath}
\usepackage{epsfig}
\usepackage{graphicx}

\usepackage{pgf,tikz}
\usepackage{braids}
\usepackage{caption}
\usepackage{shuffle}
\usetikzlibrary{arrows}
\usetikzlibrary{decorations.pathreplacing}
\title[On MZV relations in knot theory]
{On relations among multiple zeta values \\ obtained in knot theory}
\author{Hidekazu Furusho}

\address{Graduate School of Mathematics, Nagoya University, 
Chikusa-ku, Furo-cho, Nagoya, 464-8602,  Japan}

\email{furusho@math.nagoya-u.ac.jp}


\date{May 3, 2020}
\newtheorem{thm}{Theorem}[section]
\newtheorem{lem}[thm]{Lemma}

\theoremstyle{remark}

\theoremstyle{definition}
\newtheorem{defn}[thm]{Definition}
\newtheorem{rem}[thm]{Remark}

\newtheorem{eg}[thm]{Example}       

\numberwithin{equation}{section}
\numberwithin{figure}{section}

\def\orientedcircle{\unitlength.2ex
 \begin{minipage}{8\unitlength}
   \begin{tikzpicture}
     \draw[->] (0,0.125) arc (90:450:0.125);
   \end{tikzpicture}
 \end{minipage}
}

\def\unorientedcircle{\unitlength.2ex
 \begin{minipage}{8\unitlength}
   \begin{tikzpicture}
     \draw (0,0.125) arc (90:450:0.125);
   \end{tikzpicture}
 \end{minipage}
}

\begin{document}
\bibliographystyle{amsalpha+}
\maketitle

%%%%%%%%%%%%%%%%%%%%%%%%%%%%%%%%%%%%%%%%%%%%%%%%%%%%%%%%%%%%%%%%%%%%%%
\begin{abstract}
This paper focuses
linear and algebraic relations among multiple zeta values
which were obtained in knot theory.
It is shown that  they
can be derived from the associator relations,
i.e. the pentagon equation and the shuffle relation.
\end{abstract}

%\tableofcontents

%%%%%%%%%%%%%%%%%%%%%%%%%%%%%%%%%%%%%%%%%%%%%%%%%%%%%%%%%%%%%%%%%%%%%%%
\setcounter{section}{-1}
\section{Introduction}

The {\it multiple zeta value} (MZV in short)
%$\zeta(k_1,\cdots,k_m)$
is the real number defined by the following power series
\begin{equation}\label{MZV}
\zeta({\bf k})
%\zeta(k_1,\cdots,k_m)
:=\sum_{0<n_1<\cdots<n_m}\frac{1}
%{\zeta_1^{n_1}\cdots \zeta_m^{n_m}}
{n_1^{k_1}\cdots n_m^{k_m}}
\end{equation}
for ${\bf k}=(k_1,\dots,k_m)$ with $m, k_1,\dots, k_m\in {\mathbb N} (={\mathbb Z}_{>0})$
which is {\it admissible}, i.e. $k_m>1$ (its convergent condition).
The MZV's were studied allegedly first by Euler \cite{E} for $m=1$ and $2$.
%The MZV's were introduced by Euler %\cite{E}
They have recently undergone a huge revival of interest due to their appearance in various branches of mathematics and physics. 
%Various relations among MZV's have been found and studied so far.
In connection with motive theory (\cite{A, Br, DG}),
linear and algebraic relations among MZV's are particularly important.

As far as the author knows, 
there are 4 relations among  MZV's which were obtained in knot theory:

%\begin{enumerate}
%\item 
(A) {\it Le-Murakami relation (1995)} 
in \cite{LM95} Lemma 3.2.1 and Theorem 3.3.1:\\
For any integer $N\geqslant 2$ (actually $N$ can be a variable),
$$
N\frac{\sinh h}{\sinh Nh}=
%\left[
1+
%\sum_{n=1}^\infty\sum_{g=1}^{[n/2]}
%\underset{{\rm wt}({\bf k})=2w}{\sum_{{\rm ht}({\bf k})=g}}
\sum_{\substack{
                {{\bf k}:\text{ admissible}} %\\
%                {\rm wt}({\bf k})=n, \ {{\rm ht}({\bf k})=g}
                }}
\frac{(-1)^{{\rm dp}({\bf k})}(1-N^2)
(2N)^{{\rm wt}({\bf k})}}
{N^{2\cdot{\rm ht}({\bf k})}(2\pi\sqrt{-1})^{{\rm wt}({\bf k})}}
\zeta({\bf k})h^{{\rm wt}({\bf k})}
%\right]
$$
holds in $\mathbb C[[h]]$.
Here ${\rm wt}({\bf k}):=k_1+\cdots+k_m$, ${\rm dp}({\bf k}):=m$
and ${\rm ht}({\bf k}):=\sharp\{i\bigm|k_i>1\}$.

%\item 
(B) {\it Le-Murakami relation (1996)} in \cite{LM96} Theorem 4.6: \\
For any integer $N\geqslant 3$ (actually $N$ can be a variable),
\begin{align*}
%\frac{N\left(\exp(h)-\exp(-h)\right)}{\exp(N-1)h-\exp(1-N)h+\exp (h)-\exp (-h)}
\frac{N\sinh h}{\sinh(N-1)h+\sinh h}
=1+\sum_{{\bf k}:\text{ admissible}}
\frac{(-1)^{{\rm dp}({\bf k})}2^{{\rm wt}({\bf k})}g({\bf k})}
{(2\pi\sqrt{-1})^{{\rm wt}({\bf k})}}
\zeta({\bf k})h^{{\rm wt}({\bf k})}
\end{align*}
holds in $\mathbb C[[h]]$
where 
$$
g(k_1,\dots,k_m):=
(0,1,0)\cdot
%\left(
%\begin{array}{ccc}
%0&1&0\\
%\end{array}
%\right)
uv^{k_1-1}uv^{k_2-1}\cdots uv^{k_m-1}\cdot
%\left(
%\begin{array}{c}
%1\\
%1\\
%N\\
%\end{array}
%\right)
(1,1,N)^t\in {\mathbb Z}
$$
with
$$
u:=
\left(
\begin{array}{ccc}
N-1 & 1   & -1\\
1   & N-1 & -1\\
0   &   0 &  0\\
\end{array}
\right),\quad
v:=
\left(
\begin{array}{ccc}
N-1 &  -1 &  1\\
0   &   0 &  0\\
1   &  -1 & N-1\\
\end{array}
\right).
$$

%\item 
(C) {\it Takamuki relation} in \cite{T} Proposition 5, 6 and Theorem 3:\\
For any integer $N\geqslant 2$ (actually $N$ can be a variable),
\begin{align*}
N^2&\frac{\sinh h}{\sinh Nh}\cdot\frac{\cosh Nh}{\cosh h}
=N+(N^2-1)N\cdot \\
&
\sum_{n=1}^\infty\Biggl\{
\sum_{k=1}^n
\sum_{({\bf p},{\bf q},{\bf r})\in I_{2n,k}}
 (-1)^{{{\rm wt}({\bf r})}}2^{2n-k} A({\bf p},{\bf q},{\bf r};N)
\binom{{\bf p}+{\bf r}}{{\bf r}}
\frac{\zeta\bigl(\tau({\bf p}+{\bf r},{\bf q})\bigr)}{(2\pi\sqrt{-1})^{2n}} + \\
&
\sum_{\substack{2\leqslant l,m\leqslant 2n-2 \\ l+m=2n}}
\sum_{i=1}^{[l/2]}\sum_{j=1}^{[m/2]} 
\sum_{\substack{({\bf p},{\bf q},{\bf r})\in I_{l,i}\\ ({\bf s},{\bf t},{\bf u})\in I_{m,j}}}
(-1)^{{\rm wt}({\bf u})+{\rm wt}({\bf r})+m}
2^{2n-i-j-1} B({\bf p},{\bf q},{\bf r},{\bf s},{\bf t},{\bf u};N)
\\
& 
\qquad \cdot\binom{{\bf p}+{\bf r}}{{\bf r}}
\binom{{\bf s}+{\bf u}}{{\bf u}}
\frac{\zeta\bigl(\tau({\bf p}+{\bf r},{\bf q})\bigr)\zeta\bigl(\tau({\bf s}+{\bf u},{\bf t})\bigr)}{(2\pi\sqrt{-1})^{2n}}
\Biggr\}h^{2n}
\end{align*}
holds in $\mathbb C[[h]]$. 
For ${\bf p}=(p_1,\dots,p_k)$, ${\bf q}=(q_1,\dots,q_k)\in \mathbb Z_{>0}^k$ 
%($p_i,q_i\in\mathbb Z_{>0}$) and $0$ otherwise, 
we put
$$\binom{{\bf p}+{\bf q}}{\bf q}:=\prod_{i=1}^r \binom{p_i+q_i}{q_i}
\text{ and }
\tau({\bf p},{\bf q}):=({\bf 1}^{p_1-1},q_1+1,\dots,{\bf 1}^{p_k-1},q_k+1)
$$
where  ${\bf 1}^n$  means the index where $1$ repeats $n$-times.
We also put
 $$I_{n,k}:=\{({\bf p},{\bf q},{\bf r})\bigm|
{\bf p},{\bf q},{\bf r}\in {\mathbb Z}_{\geqslant 0}^k, 
{\rm wt}({\bf p}+{\bf q}+{\bf r})=n,
q_i\geqslant 1, p_i+r_i\geqslant 1, p_1\geqslant 1
\},$$
\begin{align*}
&A({\bf p},{\bf q},{\bf r};x):=x^{{\rm wt}({\bf q})-k}
\{(x-1)^{p_1}-(x+1)^{p_1}\}
\{(x-1)^{{\rm wt}({\bf r})}-(x+1)^{{\rm wt}({\bf r})}\} \\
&\qquad\qquad\qquad\qquad \cdot\prod_{a=2}^i\{(x-1)^{p_i+1}+(x+1)^{p_i+1}\},\\
&B({\bf p},{\bf q},{\bf r},{\bf s},{\bf t},{\bf u};x):=
\frac{A({\bf p},{\bf q},{\bf r};x)A({\bf s},{\bf t},{\bf u};x)
\{(x-1)^{{\rm wt}({\bf r+u})}-(x+1)^{{\rm wt}({\bf r+u})}\}}
{\{(x-1)^{{\rm wt}({\bf r})}-(x+1)^{{\rm wt}({\bf r})}\}\{(x-1)^{{\rm wt}({\bf u})}-(x+1)^{{\rm wt}({\bf u})}\}},
\end{align*}
for $({\bf p},{\bf q},{\bf r})\in I_{l,i}$ and
$({\bf s},{\bf t},{\bf u})\in I_{m,j}$.

%\item 
(D) {\it Ihara-Takamuki relation} in \cite{IT} Theorem 2:
$$
\frac{7^2\cdot [6]_q[4]_q}{[12]_q[7]_q[2]_q}
=7+\sum_{\substack{{\bf p},{\bf q} \\  {\rm dp}({\bf p})={\rm dp}({\bf q})}}
\frac{(-1)^{{\rm wt}({\bf p})}w({\bf p},{\bf q})}{(2\pi\sqrt{-1})^{{\rm wt}({\bf p})+{\rm wt}({\bf q})}}
\zeta\bigl(\tau({\bf p},{\bf q})\bigr)h^{{\rm wt}({\bf p})+{\rm wt}({\bf q})}
$$
holds in $\mathbb C[[h]]$
where
$$[n]_q:=\frac{q^{n/2}-q^{-n/2}}{q^{1/2}-q^{-1/2}}$$
for an integer $n\geqslant 1$ with $q:=e^h$ and
$$
w({{\bf p},{\bf q}}):=
-(7,7,7,7)\cdot
x^{p_1}yx^{q_1}y\cdots x^{p_k}yx^{q_k}y\cdot
(27,7,14,0)^t \in {\mathbb Q}
$$
for ${\bf p}=(p_1,\dots,p_k)$, ${\bf q}=(q_1,\dots,q_k)\in \mathbb Z_{>0}^k$ 
with
$$
x:=
\left(
\begin{array}{rrrr}
-14 &  0  &  0  &  0\\
0   & -6  &  0  &  0\\
0   &  0  & -12 &  0\\
0   &  0  &  0  &  0\\
\end{array}
\right),\quad
y:=
\left(
\begin{array}{rrrr}
\frac{5}{14} & -\frac{9}{14}  & -\frac{9}{14}& \frac{27}{7}\\
-\frac{1}{6} & -\frac{1}{2} &  \frac{1}{2} & 1\\
-\frac{1}{3} &  1 & 0 & 2\\
\frac{1}{7} &\frac{1}{7} &\frac{1}{7} &\frac{1}{7} \\
\end{array}
\right).
$$
%\end{enumerate}
Noe that (C) is an algebraic relation.

The purpose of this paper is to clarify that 
the above 4 relations can be derived from the  associator relations,
namely from the pentagon relation and the shuffle relation
(see Definition \ref{associator} and Theorem \ref{pentagon-hexagons}).

\begin{thm}\label{main theorem}
Let $(\mu,\varphi)$ be any associator (see Definition \ref{associator} below).
Then the relations replacing $\zeta(k_1,\dots,k_m)$ with
$\zeta_{\varphi}(k_1,\dots,k_m)$ and
$2\pi\sqrt{-1}$ with $\mu$
in (A)-(D) hold.
\end{thm}

Here for  each series $\varphi\in \mathbb C\langle\langle X_0,X_1\rangle\rangle$
we denote the coefficient of 
$X_0^{k_m-1}X_1\cdots X_0^{k_1-1}X_1$
($m$, $k_1$,\dots, $k_m\in {\mathbb N}$ and $k_m>1$)
in $\varphi$
multiplied with $(-1)^m$ by 
$$\zeta_{\varphi}(k_1,\dots,k_m)\in\mathbb C.$$

The contents of this paper go as follows:
\S \ref{sec:associators} is a review of the formalism of
associators and the Grothendieck-Teichm\"{u}ller group.
In \S \ref{sec:lemmas} we show auxiliary lemmas to prove
Theorem \ref{main theorem} in  
\S \ref{sec:proof of main theorem}.

%%%%%%%%%%%%%%%%%%%%%%%%%%%%%%%%%%%%%%%%%%%%%%%%%%%%%%%%%%%%%%%%%%%%%%
\section{Associators and the Grothendieck-Teichm\"{u}ller group} \label{sec:associators}
Denote by
$U\frak F_2=\mathbb C\langle\langle X_0,X_1\rangle\rangle$
the non-commutative formal power series ring, which is regarded as
the completion of
the universal enveloping algebra of the  free Lie algebra
$\frak F_2$ with two variables $X_0$ and $X_1$.
It is equipped with a structure of Hopf algebra
%whose product $\circ: U\frak F_2\hat\otimes U\frak F_2 \to U\frak F_2$
%is given by the shuffle product and
whose coproduct
$\Delta: U\frak F_2\to U\frak F_2\hat\otimes U\frak F_2$
is given by 
$$
\Delta(X_0)=X_0\otimes 1+1\otimes X_0 \text{ and }
\Delta(X_1)=X_1\otimes 1+1\otimes X_1,
$$
whose unit
$\epsilon: U\frak F_2\to U\frak F_2\hat\otimes U\frak F_2$
is given by 
$$
\epsilon(X_0)=0 \text{ and }
\epsilon(X_1)=0,
$$
and
whose antipode 
$S:U\frak F_2\to U\frak F_2$ is the {\it anti}-automorphism
such that
$$
S(X_0)=-X_0 \text{ and } S(X_1)=-X_1.
$$
For any algebra homomorphism $\iota:U\frak F_2\to S$,
the image $\iota(\varphi)\in S$ is denoted 
by $\varphi(\iota(X_0),\iota(X_1))$.

\begin{defn}[\cite{Dr}]\label{associator}
A pair $(\mu,\varphi)$ with a {\it non-zero} element $\mu$ in $\mathbb C$ and a series 
$\varphi=\varphi(X_0,X_1)\in U\frak F_2$ is called an {\it associator}
if it satisfies the {\it associator relations}  \eqref{shuffle}--\eqref{hexagon-b},
that is, 
the {\it shuffle product} 
\footnote{See its equivalent formulation \eqref{shuffle2}.
It is also known as the group-like condition because it is equivalent to
saying $\varphi\in\exp\hat{\frak F}_2$.
}
\begin{equation}\label{shuffle}
\Delta (\varphi)=\varphi\otimes \varphi
\text{ and } 
\varphi(0,0)=1,
\end{equation}
{\it one pentagon equation} 
\begin{equation}\label{pentagon}
\varphi(t_{12},t_{23}+t_{24})
\varphi(t_{13}+t_{23},t_{34})=
\varphi(t_{23},t_{34})
\varphi(t_{12}+t_{13},t_{24}+t_{34})
\varphi(t_{12},t_{23})
\end{equation}
in $U\frak a_4$
and
{\it two hexagon equations}
\begin{equation}\label{hexagon}
\exp\{\frac{\mu (t_{13}+t_{23})}{2}\}=
\varphi(t_{13},t_{12})\exp\{\frac{\mu t_{13}}{2}\}
\varphi(t_{13},t_{23})^{-1}
\exp\{\frac{\mu t_{23}}{2}\} \varphi(t_{12},t_{23}),
\end{equation}
\begin{equation}\label{hexagon-b}
\exp\{\frac{\mu (t_{12}+t_{13})}{2}\}=
\varphi(t_{23},t_{13})^{-1}\exp\{\frac{\mu t_{13}}{2}\}
\varphi(t_{12},t_{13})
\exp\{\frac{\mu t_{12}}{2}\} \varphi(t_{12},t_{23})^{-1}
\end{equation}
in $U\frak a_3$
where we denote by $U\frak a_3$ (resp. $U\frak a_4$) 
the completion of the universal enveloping algebra of
the {\it pure braid Lie algebra} $\frak a_3$ (resp. $\frak a_4$)
with 3 (resp. 4) strings,
which is generated by $t_{ij}$ ($1\leqslant i,j \leqslant 3$ (resp. $4$))
with defining relations
$$t_{ii}=0, \  t_{ij}=t_{ji}, \ [t_{ij},t_{ik}+t_{jk}]=0
\text{ ($i$, $j$, $k$: all distinct)}$$
$$\text{and  }\ [t_{ij},t_{kl}]=0
\text{ ($i$, $j$, $k$, $l$: all distinct).}$$
\end{defn}

%Note that $X_0\mapsto t_{12}$ and $X_1\mapsto t_{23}$
%give an isomorphism $U\frak F_2\simeq U\frak a_3$. %%%%%That is a mistake!
%\begin{rem}
%(i).
%Actually Drinfeld \cite{Dr} proved that such a pair always exists for any field $k$ of
%characteristic $0$.

%(ii).
The equations \eqref{pentagon}--\eqref{hexagon-b} reflect the three axioms of 
braided monoidal categories %introduced by Joyal and Street 
\cite{JS}.
%We note that for any $k$-linear {\it infinitesimal} tensor category $\mathcal C$, 
%each associator  gives a structure of 
%a braided  %(actually ribbon) 
%monoidal category on 
%${\mathcal C}[[h]]$ (cf. \cite{C, Dr, KT}).
%Here ${\mathcal C}[[h]]$ denotes the category whose set of objects is equal to that of
%${\mathcal C}$ and whose set of morphisms $Mor_{{\mathcal C}[[h]]}(X,Y)$ is
%$Mor_{\mathcal C}(X,Y)\otimes k[[h]]$ ($h$: a formal parameter).
%
%(iii).
Associators are essential for construction of quasi-triangular quasi-Hopf
quantized universal enveloping algebras (cf. \cite{Dr})
and also for a 
%(iv).
reconstruction of
universal Vassiliev knot invariant (the Kontsevich invariant \cite{Ko,Bar95})
in a combinatorial way (cf.
Le and Murakami \cite{LM96b}, Bar-Natan \cite{BN}, 
Kassel and Turaev \cite{KT}).
%(see also \cite{C}).
%\end{rem}

Drinfeld \cite{Dr} proved that such a pair  $(\mu,\varphi)$ always exists for any 
given $\mu$. % \in\mathbb C^\times$.
Note that 
$\left(1,\varphi(\frac{X_0}{\mu},\frac{X_1}{\mu})\right)$ is an associator
if and only if $(\mu,\varphi)$ is so.
Any asociator satisfies the so-called
%We remind that it was shown in  \cite{F10} Lemma 6 that \eqref{pentagon} follows
{\it 2-cycle relation} (cf. \cite{Dr}).
\begin{equation}\label{2-cycle}
\varphi(X_0,X_1)\varphi(X_1,X_0)=1,
\end{equation}
which is a consequence of \eqref{shuffle} and \eqref{pentagon}
(cf. \cite{F10}).
Actually the two hexagon equations are a consequence of
the one pentagon equation:

\begin{thm}[\cite{F10}]\label{pentagon-hexagons}
Let $\varphi=\varphi(X_0,X_1)$ be %a commutator group-like 
an element of $U\frak F_2$ satisfying \eqref{shuffle} and
%Suppose that $\varphi$ satisfies?
\eqref{pentagon}.
Then there always exists $\mu\in\mathbb C$ (unique up to signature)
such that the pair $(\mu,\varphi)$ satisfies
\eqref{hexagon} and \eqref{hexagon-b}.
\end{thm}

Several different proofs of the above theorem were obtained
in \cite{AT, BD, W}.
For various aspects of associators, consult \cite{F14}.

The Grothendieck-Teichm\"{u}ller group was introduced by Drinfeld \cite{Dr}
in his study of deformations of quasi-triangular quasi-Hopf
quantized universal enveloping algebras.
It was defined to be the set of degenerated associators.
The construction of the group was also stimulated  by the previous idea of Grothendieck,
{\it un jeu de Teichm\"{u}ller-Lego},
which was posed in his research proposal
{\it Esquisse d'un programme} \cite{G}.

\begin{defn}[\cite{Dr}]
A {\it degenerated associator}  is a group-like series $\varphi\in U\frak F_2$ %(i.e. $\varphi\in\underline{F_2}(k)$)
%with $c_{X_0}(\varphi)=c_{X_1}(\varphi)=c_{X_0X_1}(\varphi)=0$
satisfying the defining equations \eqref{pentagon}--\eqref{hexagon-b}
of associators with $\mu=0$.
The %(unipotent part of the graded) 
{\it Grothendieck-Teichm\"{u}ller} %(pro-algebraic) 
{\it group}
\footnote{
It is denoted by $GRT_1$ because it is a {\it unipotent} part of 
the {\it graded} version of the
Grothendieck-Teichm\"{u}ller group. % $GT$.
}
$GRT_1(\mathbb C)$
is defined %by ${\underline M}\backslash M$, that is,
to be %the pro-algebraic variety whose 
the set of %$k$-valued points consists of
degenerated associators.
%which are
%group-like series $\varphi\in U\frak F_2$ %(i.e. $\varphi\in\underline{F_2}(k)$)
%%with $c_{X_0}(\varphi)=c_{X_1}(\varphi)=c_{X_0X_1}(\varphi)=0$
%satisfying the defining equations \eqref{pentagon}--\eqref{hexagon-b}
%of associators with $\mu=0$.
\end{defn}

%\begin{rem}\label{GRT-group}
%\begin{enumerate}
%\renewcommand{\labelenumi}{(\roman{enumi})}
%\item
%By Theorem \ref{pentagon-hexagons},
%$GRT_1$ is reformulated to be the set of
%group-like series satisfying \eqref{pentagon} without quadratic terms.
%
%\item
It indeed forms a %pro-unipotent algebraic
group \cite{Dr}
by the multiplication
\begin{equation}\label{multiplication}
\varphi_1\circ\varphi_2:%=\varphi_1(\varphi_2 X_0\varphi_2^{-1},X_1)\cdot\varphi_2
=\varphi_2\cdot\varphi_1(X_0,\varphi^{-1}_2X_1\varphi_2).
\end{equation}
It was also shown that %$M_{\mu_0}(\mathbb C)$,
the set of associators with each fixed $\mu\in\mathbb C^\times$
forms a right $GRT_1(\mathbb C)$-torsor by \eqref{multiplication}.
%for each $\mu_0\in\mathbb C^\times$.

%By the map $X_0\mapsto X_0$ and $X_1\mapsto \varphi^{-1} X_1\varphi$,
%the group $GRT_1$ is regarded as a subgroup of ${Aut}\exp{\frak F_2}$.
%
%\item
%Ihara came to the Lie algebra of $GRT_1$
%%, which he calls the stable derivation algebra, 
%independently
%of Drinfeld's work in his arithmetic study of  Galois action on fundamental groups
%(cf. \cite{I90}).
%
%\item
%The  cyclotomic analogues of associators and that of the Grothendieck-Teichm\"{u}ller group were introduced by Enriquez \cite{E}.
%Some elimination results  on their defining equations in special case were obtained 
%in \cite{EF}.
%\end{enumerate}
%\end{rem}

\begin{eg}[\cite{Dr}]
By using the KZ (Knizhnik-Zamolodchikov) equation, Drinfeld constructed
a series $\varPhi_{\mathrm{KZ}}=\varPhi_{\mathrm{KZ}}(X_0,X_1)\in U\frak F_2$
%\in{\bold C}\langle\langle X_0,X_1\rangle\rangle$
called the {\it KZ associator} (a.k.a Drinfeld associator)
and he showed that the pair $(2\pi\sqrt{-1},\varPhi_{\mathrm{KZ}})$
satisfies the associator relations \eqref{shuffle}--\eqref{hexagon-b} in \cite{Dr}.

One of the most important properties of  $\varPhi_{\mathrm{KZ}}$ is that 
MZV's appear as its coefficients:
\begin{equation}\label{KZ=zeta}
\zeta_{\varPhi_{\mathrm{KZ}}}(k_1,\cdots,k_m)= \zeta(k_1,\cdots,k_m)
\end{equation}
for $k_m>1$.
Actually general coefficients of $\varPhi_{\mathrm{KZ}}$ are calculated 
to be linear combinations of MZV's
by the formula \eqref{regularization} below.
\footnote{An essentially same formula appeared in \cite{LM96} Theorem A.9
though it  seems to include an error on the signature
which was corrected in \cite{F03} Proposition 3.2.3.
}

Since $\varPhi_{\mathrm{KZ}}$ satisfies the associator relations
and MZV's appear as its coefficients,
lots of algebraic relations among MZV's are obtained.
\end{eg}

%%%%%%%%%%%%%%%%%%%%%%%%%%%%%%%%%%%%%%%%%%%%%%%%%%%%%%%%%%%%%%%%%%%%%%
\section{Auxiliary lemmas} \label{sec:lemmas}
We denote $\shuffle: U\frak F_2\hat\otimes U\frak F_2 \to U\frak F_2$
to be a {\it shuffle product} of $U\frak F_2$,
which is an associative and commutative product % of $U\frak F_2$
recursively defined by
%\begin{itemize}
%\item 
$W\shuffle 1=1\shuffle W=W$ and
%\item 
$$UW\shuffle VW'=U(W\shuffle VW')+V(UW\shuffle W') \text{ with } U,V\in\{A,B\}.$$
for any word (a monic monomial element in  $U\frak F_2$)
$W$ and $W'$.
%\end{itemize}
It can be identified with the dual of the coproduct $\Delta:U\frak F_2\to U\frak F_2\hat\otimes U\frak F_2$
by the identification $U\frak F_2^*\simeq U\frak F_2$.
For each series $\varphi\in U\frak F_2$, we
denote the linear map sending each word %(a monic monomial element in  $U\frak F_2$) 
to
its coefficient %of each word %(a monic monomial element in $U\frak F_2$) $W$ 
in $\varphi$ by %$I_\varphi(W)$.
$I_\varphi:U\frak F_2\to \mathbb C$,
whence 
$$
I_\varphi(X_0^{k_m-1}X_1\cdots X_0^{k_1-1}X_1)=(-1)^m  \zeta_\varphi(k_1,\cdots,k_m)
$$
for $k_m>1$.
We note that \eqref{shuffle} for $\varphi$ is equivalent to
\begin{equation}\label{shuffle2}
I_\varphi(W)\cdot I_\varphi(W')=I_\varphi(W\shuffle W')
\text{ and } I_\varphi(1)=1
\end{equation}
for any $W$ and $W'\in U\frak F_2$.

\begin{lem}%[\cite{F03} Proposition 3.2.3.]
\label{Le-Murakami method}
Let $\varphi\in U\frak F_2$ be group-like without linear terms
i.e. a series satisfying \eqref{shuffle} and $I_\varphi(X_0)=I_\varphi(X_1)=0$.
Suppose that $W$ is written as $X_1^rVX_0^s$ ($r,s\geqslant 0$, 
$V\in X_0 \cdot U\frak F_2\cdot X_1$ or $V=1$).
Then
\begin{equation}\label{regularization}
I_\varphi(W)=\sum_{\substack{0\leqslant a\leqslant r\\ 0\leqslant b\leqslant s}}(-1)^{a+b}
I_\varphi\Bigl(\pi(X_1^a\shuffle X_1^{r-a}VX_0^{s-b}\shuffle X_0^b)\Bigr).
\end{equation}
Here $\pi:U\frak F_2\to U\frak F_2$ is
%{\mathbb C}+X_0 \cdot U\frak F_2\cdot X_1$ is the composition of 
the natural projection 
$U\frak F_2\to {\mathbb C}+X_0 \cdot U\frak F_2\cdot X_1
(\subset U\frak F_2)$
annihilating $X_1\cdot U\frak F_2$ and $U\frak F_2\cdot X_0$.
%with the natural inclusion to $U\frak F_2$.
\end{lem}

\begin{proof}
Let $x$ and $y$ be commutative variables and consider 
$$\mathbb C\langle\langle X_0,X_1\rangle\rangle [[x_0,x_1]]:=
\mathbb C\langle\langle X_0,X_1\rangle\rangle \widehat{\otimes}
\mathbb C[[x_0,x_1]].$$
Let
$$g_1:\mathbb C\langle\langle X_0,X_1\rangle\rangle \to
\mathbb C\langle\langle X_0,X_1\rangle\rangle [[x_0,x_1]]$$
be the algebra homomorphism sending $X_0$, $X_1$ to 
$$X_0-x_0:=X_0\otimes 1- 1\otimes x_0, \quad
X_1-x_1:=X_1\otimes 1- 1\otimes x_1$$
respectively and
let 
$$g_2:\mathbb C\langle\langle X_0,X_1\rangle\rangle [[x_0,x_1]]\to
\mathbb C\langle\langle X_0,X_1\rangle\rangle $$
be the well-defined $\mathbb C$-linear map sending $W\otimes x_0^px_1^q$
to $X_1^qWX_0^p$ for each word $W$ and $p,q\geqslant 0$.
Then by definition
$$g_2\circ g_1(VX_0)=g_2\circ g_1(X_1V)=0$$ 
for any 
$V\in \mathbb C\langle\langle X_0,X_1\rangle\rangle$.
So we have
\begin{equation*}
g_2\circ g_1\circ \pi=g_2\circ g_1
\end{equation*}
%The proof can be done in the same way to 
(cf. the arguments in \cite{F04} Lemmas 3.32-3.38. and in \cite{LM96} Appendix A.)

By \eqref{shuffle} and our assumption, we have
$$\varphi(X_0-x_0,X_1-x_1)=\varphi(X_0,X_1)$$
in
$\mathbb C\langle\langle X_0,X_1\rangle\rangle [[x_0,x_1]]$.
Namely 
\begin{equation*}
\varphi=(g_2\circ g_1)(\varphi).
\end{equation*}
By combining the above two equations, we get
$$
\varphi=(g_2\circ g_1\circ \pi)(\varphi).
$$
By comparing its coefficient of $W=X_1^rVX_0^s$, %in the right hand side is given as 
we obtain \eqref{regularization}.
\end{proof}

The following lemma will be used in the proof of our main theorem.

\begin{lem}\label{lem:duality}
Assume that $\varphi\in U\frak F_2$ satisfies \eqref{shuffle} and \eqref{2-cycle}.
Then the duality relation
\begin{equation}\label{duality}
\zeta_\varphi\bigl(\tau({\bf p},{\bf q})\bigr)=
\zeta_\varphi\bigl(\tau({\bf q^*},{\bf p^*})\bigr)
\end{equation} 
holds for any
${\bf p}=(p_1,\dots,p_k)$, ${\bf q}=(q_1,\dots,q_k)\in \mathbb Z_{>0}^k$ 
where ${\bf r^*}:=(r_k,\dots,r_1)$ for ${\bf r}=(r_1,\dots,r_k)$.
\end{lem}

\begin{proof}
By defninitoin,
$$\zeta_\varphi\bigl(\tau({\bf p},{\bf q})\bigr)
=(-1)^{{\rm wt}({\bf p})}I_\varphi(X_0^{q_k}X_1^{p_k}\cdots X_0^{q_1}X_1^{p_1}).
$$
By the fundamental property of universal enveloping algebra,
$S(x)=-x$ for $x\in\frak F_2$, hence
%\begin{equation}\label{S}
$$S(\varphi)=\varphi^{-1}.$$
%\end{equation}
Then by combining it with \eqref{2-cycle}, we obtain \eqref{duality}.
\end{proof}

%\begin{rem}
%(i).
%The Drinfeld associator is expressed as follows:
%\begin{align*}
%\varPhi_{\mathrm{KZ}}(X_0,X_1)=1+
%\underset{k_m>1}
%{\underset{m,k_1,\dots,k_m\in{\bf N}}{\sum}}
%(-1)^m & \zeta(k_1,\cdots,k_m)
%X_0^{k_m-1}X_1\cdots X_0^{k_1-1}X_1 \\
%&+\text{(regularized terms)}.
%\end{align*}
%Here $\zeta(k_1,\cdots,k_m)$ is the
%{\it multiple zeta value} (MZV in short),
%the real number defined by the following power series
%\begin{equation}\label{MZV}
%\zeta(k_1,\cdots,k_m)
%:=\sum_{0<n_1<\cdots<n_m}\frac{1}
%%{\zeta_1^{n_1}\cdots \zeta_m^{n_m}}
%{n_1^{k_1}\cdots n_m^{k_m}}
%\end{equation}
%for $m$, $k_1$,\dots, $k_m\in {\bf N} (={\bf Z}_{>0})$
%with $k_m>1$ (its convergent condition).
%All of the coefficients of $\varPhi_{\mathrm{KZ}}$ (including its regularized terms)
%are explicitly calculated in terms of
%MZV's  in \cite{F03} Proposition 3.2.3
%by Le-Murakami's method in \cite{LMb}.
%
%(ii).
%Since all of the coefficients of $\varPhi_{\mathrm{KZ}}$ are described by MZV's,
%the equations \eqref{shuffle}$\sim$\eqref{hexagon-b} for
%$(\mu,\varphi)=(2\pi\sqrt{-1},\varPhi_{\mathrm{KZ}})$ yield
%algebraic relations among them,
%which are called {\it associator relations}.
%It is expected that the associator relations might
%produce all algebraic relations among MZV's.
%\end{rem}

%%%%%%%%%%%%%%%%%%%%%%%%%%%%%%%%%%%%%%%%%%%%%%%%%%%%%%%%%%%%%%%%%%%%%%
\section{Proof of Theorem \ref{main theorem}}\label{sec:proof of main theorem}
Let $X$ be a $1$-dimensional compact oriented real manifold 
with ordered connected components.
Denote ${\mathcal A}(X)$ to be a $\mathbb C$-vector space of chord diagrams there.
A finite number (which yields a degree in ${\mathcal A}(X)$)
of unordered pairs of distinct interior points on $X$
regarded up to  orientation and component preserving homeomorphisms
subject to 4T-relation.
We denote its completion with respect to the degree  by the same symbol
${\mathcal A}(X)$.
Put ${\mathcal A}_0(X)$ to be its further quotient by FI-relation. 
Here the  {\it  4T-relation} stands for the 4 terms relation
defined by 
$$
D_1-D_2+D_3-D_4=0
$$
where $D_j$ are chord diagrams with four chords
identical outside a ball in which they differ as illustrated in Figure \ref{4T-relation}.
%%%%%%%%%%%%%%%%%%%%%%%%%%%%%%%%%%%%%%%%%%%%%%%%%%%%%%%
\begin{figure}[h]
\begin{center}
            \begin{tikzpicture}
                  \draw[->] (0,0)--(0,1.5) ;
                  \draw[->] (0.5,0)--(0.5,1.5) ;
                  \draw[->] (1,0)--(1,1.5) ;
                   \draw[densely dotted] (0,0.5)--(1,0.5);
                   \draw[densely dotted] (0,1)--(0.5,1);
                   \draw(0.5,-0.5) node{$D_1$};
                   \draw(1.5,0.7) node{$-$};
         \begin{scope}[xshift=2cm]
                  \draw[->] (0,0)--(0,1.5) ;
                  \draw[->] (0.5,0)--(0.5,1.5) ;
                  \draw[->] (1,0)--(1,1.5) ;
                   \draw[densely dotted] (0,0.5)--(0.5,0.5);
                   \draw[densely dotted] (0,1)--(1,1) ;
                   \draw(0.5,-0.5) node{$D_2$};
                   \draw(1.5,0.7) node{$+$};
         \end{scope}                   
         \begin{scope}[xshift=4cm]
                  \draw[->] (0,0)--(0,1.5) ;
                  \draw[->] (0.5,0)--(0.5,1.5) ;
                  \draw[->] (1,0)--(1,1.5) ;
                   \draw[densely dotted] (0.5,0.5)--(1,0.5);
                   \draw[densely dotted] (0,1)--(0.5,1);  
                   \draw(0.5,-0.5) node{$D_3$};
                   \draw(1.5,0.7) node{$-$};
         \end{scope}
         \begin{scope}[xshift=6cm]
                  \draw[->] (0,0)--(0,1.5) ;
                  \draw[->] (0.5,0)--(0.5,1.5) ;
                  \draw[->] (1,0)--(1,1.5) ;
                   \draw[densely dotted] (0,0.5)--(0.5,0.5);
                   \draw[densely dotted] (0.5,1)--(1,1);
                 \draw(0.5,-0.5) node{$D_4$};
                 \draw(1.5,0.7) node{$=0$};
         \end{scope}
            \end{tikzpicture}
%            \end{center}
    \caption{4T-relation}
    \label{4T-relation}
\end{center}
\end{figure}
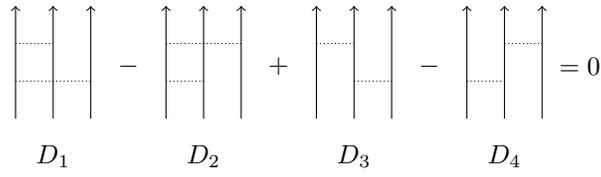
%%%%%%%%%%%%%%%%%%%%%%%%%%%%%%%%%%%%%%%%%%%%%%%%%%%%%%%

The  {\it  FI-relation} stands for the frame independent relation
where we put 
$$
D=0
$$
for any chord diagrams $D$ with an isolated chord as illustrated in Figure \ref{FI-relation}.
%a chord that does not intersect on any other one in  their diagrams.
In more detail for those notions, consult \cite{Bar95, Ko}. %for example.
%%%%%%%%%%%%%%%%%%%%%%%%%%%%%%%%%%%%%%%%%%%%%%%%%%%%%%%
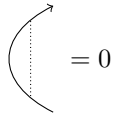
\begin{figure}[h]
  \begin{center}
       \begin{tikzpicture}
            \draw[<-](0,1) arc(135:225: 2cm and 1cm);
            \draw[densely dotted] (-0.3,-0.2)--(-0.3,0.8);
            \draw(0.5,0.3) node{$=0$};         
            \draw(0.5,1.5) node{};
            \draw(0.5,-0.8) node{};
       \end{tikzpicture}
         \caption{FI-relation}
         \label{FI-relation}
  \end{center}
\end{figure}

{\bf Proof of Theorem \ref{main theorem}:}
Let $\gamma_{\mathrm{KZ}}$ be the element in ${\mathcal A}(S^1)$ depicted as in Figure \ref{gamma}
where $S^1$ means the oriented circle.
%%%%%%%%%%%%%%%%%%%%%%%%%%%%%%%%%%%%%%%%%%%%%%%%%%%%%%%
\begin{figure}[h]
\begin{center}
\begin{tikzpicture}[xscale=1.2, yscale=1.2]
\foreach \x in {-3,-1,1,3}
\draw (\x,0)--(\x,2);
\draw[thin,fill=white] (-1.8,0.5) rectangle (3.8,1.5);
\draw (1,1) node{$\displaystyle \varPhi_{\mathrm{KZ}} \left(\frac{-1}{2\pi\sqrt{-1}}\raisebox{-1ex}{
\begin{tikzpicture}[thick]
\draw[->] (0,0)--(0,0.6);
\draw[densely dotted](0,0.3)--(0.3,0.3);
\draw[<-] (0.3,0)--(0.3,0.6);
\draw[->] (0.6,0) --(0.6,0.6);
\end{tikzpicture}
},\quad \frac{-1}{2\pi\sqrt{-1}}
\raisebox{-1ex}{
\begin{tikzpicture}[thick]
\draw[->] (-0.3,0) --(-0.3,0.6);
\draw[<-] (0,0)--(0,0.6);
\draw[densely dotted](0,0.3)--(0.3,0.3);
\draw[->] (0.3,0)--(0.3,0.6);
\end{tikzpicture}
}\right)$};
\draw[<-](-1,0) .. controls (-1,-1) and (1,-1) .. (1,0);
\draw[->](-3,0) .. controls (-3,-2.5) and (3,-2.5) .. (3,0);
\draw[<-](1,2) .. controls (1,3) and (3,3) .. (3,2);
\draw[<-](-3,2) .. controls (-3,3) and (-1,3) .. (-1,2);
\end{tikzpicture}
\caption{$\gamma_{\mathrm{KZ}}$}
\label{gamma}
\end{center}
\end{figure}
%%%%%%%%%%%%%%%%%%%%%%%%%%%%%%%%%%%%%%%%%%%%%%%%%%%%%%%

By \eqref{KZ=zeta},
its image $\bar\gamma_{\mathrm{KZ}}$ under the projection
${\mathcal A}(S^1)\twoheadrightarrow{\mathcal A}_0(S^1)$ is described below (cf. \cite{LM95} Theorem 3.1.3):
\begin{equation}\label{eq:gamma}
\bar\gamma_{\mathrm{KZ}}=\orientedcircle+
\sum_{\substack{
                {{\bf k}:\text{ admissible}}
                }}
\frac{(-1)^{{\rm wt}({\bf k})+{\rm dp}({\bf k})}}{(2\pi\sqrt{-1})^{{\rm wt}({\bf k})}}
\zeta({\bf k})D_{\bf k}.
\end{equation}
Here $D_{(k_1,\dots,k_m)}$ is the chord diagram depicted in Figure \ref{Dk}.
%%%%%%%%%%%%%%%%%%%%%%%%%%%%%%%%%%%%%%%%%%%%%%%%%%%%%%%
\begin{figure}[h]
\begin{center}
\begin{tikzpicture}[xscale=1.1]
\foreach \x in {-1.4,0,1,2}
\draw (\x,0)--(\x,4);
\draw[densely dotted] (1,0.1)--(2,0.1);

\draw[densely dotted] (0,0.2)--(1,0.2);
\draw[densely dotted] (0,0.3)--(1,0.3);
\draw[dotted] (0.5,0.2)--(0.5,0.6);
\draw[densely dotted] (0,0.5)--(1,0.5);
\draw[densely dotted] (0,0.6)--(1,0.6);
 \draw[decorate,decoration={brace}] (-0.1,0.2) -- (-0.1,0.6) node[left, midway]{$k_1-1$};

\draw[densely dotted] (1,0.7)--(2,0.7);

\draw[densely dotted] (0,0.8)--(1,0.8);
\draw[densely dotted] (0,0.9)--(1,0.9);
\draw[dotted] (0.5,0.8)--(0.5,1.2);
\draw[densely dotted] (0,1.1)--(1,1.1);
\draw[densely dotted] (0,1.2)--(1,1.2);
 \draw[decorate,decoration={brace}] (-0.1,0.8) -- (-0.1,1.2) node[left, midway]{$k_2-1$};

\draw[densely dotted] (1,1.3)--(2,1.3);

\draw[dotted] (0.5,1.7)--(0.5,3);
\draw[dotted] (1.5,1.7)--(1.5,3);

\draw[densely dotted] (1,3.4)--(2,3.4);

\draw[densely dotted] (0,3.5)--(1,3.5);
\draw[densely dotted] (0,3.6)--(1,3.6);
\draw[dotted] (0.5,3.5)--(0.5,3.9);
\draw[densely dotted] (0,3.8)--(1,3.8);
\draw[densely dotted] (0,3.9)--(1,3.9);
 \draw[decorate,decoration={brace}] (-0.1,3.5) -- (-0.1,3.9) node[left, midway]{$k_m-1$};

\draw[<-](0,0) .. controls (0,-1) and (1,-1) .. (1,0);
\draw[->](-1.4,0) .. controls (-1.4,-2.5) and (2,-2.5) .. (2,0);
\draw[<-](1,4) .. controls (1,5) and (2,5) .. (2,4);
\draw[<-](-1.4,4) .. controls (-1.4,5) and (0,5) .. (0,4);

\end{tikzpicture}
\caption{$D_{(k_1,\dots,k_m)}$}
\label{Dk}
\end{center}
\end{figure}
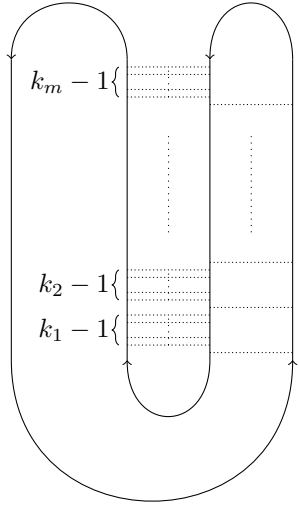
%%%%%%%%%%%%%%%%%%%%%%%%%%%%%%%%%%%%%%%%%%%%%%%%%%%%%%%

\bigskip
(A) {\it Le-Murakami relation (1995):}
They constructed a weight system (denoted $W_r$ in \cite{LM95} \S 3.2), 
 a degree preserving linear map
$$W^{95}_N:{\mathcal A}_0(S^1)\to\mathbb{C}[[h]]$$
by using the fundamental representation of the Lie algebra $\frak{sl}(N)$
for each $N\geqslant 2$.
In \cite{LM95} Lemma 3.2.1,
it was shown that  
\begin{equation}
W_N^{95}(D_{\bf k})=\frac{1-N^2}{N^{2\cdot{\rm ht}({\bf k})-1}}(-2Nh)^{{\rm wt}({\bf k})}.
\end{equation}

In \cite{LM95} Theorem 2.3.1, they showed that 
\begin{equation}\label{W=P}
N^{-2}W_N^{95}(\bar\gamma_{\mathrm{KZ}})W_N^{95}\left(\hat{Z}(K)\right)
=P_K(e^h-e^{-h},e^{-Nh})
\end{equation}
for any link $K$,
where  $\hat{Z}(K)\in {\mathcal A}_0(S^1) $ is the Kontsevich invariant
(\cite{Ko})
and $P_K(a,z)$ is the HOMFLY-PT polynomial,
the Laurent polynomial in two variables $a$ and $z$
with integer coefficients uniquely defined by the skein relation
$$
z^{-1}P_{K_+}(a,z)-zP_{K_-}(a,z)=aP_{K_0}(a,z)
$$
and the initial condition $P_{\orientedcircle}(a,z)=1$.
Here $K_+$, $K_-$ , $K_0$ are all identical except within a ball
in which they differ as illustrated below
%$$
%\underset{K_+}{\huge{\diaCrossP}} 
%\qquad\quad \underset{K_-}{\huge{\diaCrossN}}
%\qquad\quad\underset{K_0}{\huge{\diaSmooth}}
%$$
%%%%%%%%%%%%%%%%%%%%%%%%%%%%%%%%%%%%%%%%%%%%%%%%%%%%%%%
%% positive crossing
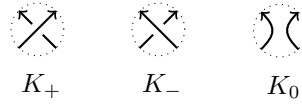
\begin{center}
    \begin{tikzpicture}%(15,15)
       \draw[->,thick] (0.5,0)--(0,0.5) ;
       \draw[line width=0.2cm, white](0,0)--(0.5,0.5) ;
       \draw[->,thick] (0,0)--(0.5,0.5) ;
       \draw[dotted] (0.25,0.25) circle (0.35);
       \draw(0.3,-0.5) node{$K_{+}$};
    \end{tikzpicture}
  \qquad
%% negative crossing
    \begin{tikzpicture}(15,15)
       \draw[->,thick] (0,0)--(0.5,0.5) ;
       \draw[line width=0.2cm, white](0.5,0)--(0,0.5) ;
       \draw[->,thick] (0.5,0)--(0,0.5) ;
       \draw[dotted] (0.25,0.25) circle (0.35);
       \draw(0.3,-0.5) node{$K_{-}$};
    \end{tikzpicture}
\qquad
%% smoothing
    \begin{tikzpicture}(15,15)
         \draw[->,thick](0,0) .. controls (0.2,0.2) and (0.2,0.3)  .. (0,0.5);
         \draw[->,thick](0.5,0) .. controls (0.3,0.2) and (0.3,0.3)  .. (0.5,0.5);
         \draw[dotted] (0.25,0.25) circle (0.35);
         \draw(0.3,-0.5) node{$K_{0}$};
    \end{tikzpicture}
    \end{center}
%}
%%%%%%%%%%%%%%%%%%%%%%%%%%%%%%%%%%%%%%%%%%%%%%%%%%%%%%%
and $\orientedcircle$ means the trivial knot.
For the $2$-components trivial link $K=\orientedcircle\,\orientedcircle$,
we have 
\begin{equation*}\label{homfly for 2-components trivial link}
P_{\orientedcircle\,\orientedcircle}(a,z)=\frac{1}{a}\{\frac{1}{z}-z\}
\end{equation*}
by the above skein relation. 
While the left hand side of \eqref{W=P} for
$K=\orientedcircle\,\orientedcircle$ was calculated to be
%$$W_N^{95}(\hat{Z}(\orientedcircle\,\orientedcircle))=W_N^{95}(\bar\gamma_{\mathrm{KZ}})^{-1}$$
$$
N^{-2}W^{95}_N(\bar\gamma_{\mathrm{KZ}})W_N^{95}(\hat{Z}(\orientedcircle\,\orientedcircle))
=N^{-2}W^{95}_N(\bar\gamma_{\mathrm{KZ}})W^{95}_N(\bar\gamma_{\mathrm{KZ}}^{-1})^2
=N^2W^{95}_N(\bar\gamma_{\mathrm{KZ}})^{-1}
%N^2/W_N^{95}(\hat{Z}(\orientedcircle\,\orientedcircle))
$$
by $\hat{Z}(\orientedcircle)=\bar\gamma_{\mathrm{KZ}}^{-1}$ and
\cite{LM95} Proposition 2.1.3
(or it can be directly derived from \cite{LM95} Theorem 2.1.4).
%By \eqref{W=P} and \eqref{homfly for 2-components tirvial link}, 
Thus the following can be deduced from  \eqref{W=P}
(cf. \cite{LM95} Theorem 3.3.1)
\begin{equation}\label{eqn:WN-LM95}
W^{95}_N(\bar\gamma_{\mathrm{KZ}})=N^2\sinh h/\sinh Nh.
\end{equation}
%by relating $W^{95}_N$ with the HOMFLY-PT knot polynomial (cf. \cite{LM95} Theorem 2.1.4 and 2.3.1).
This is how the relation was shown for 
$(2\pi\sqrt{-1},\varPhi_{\mathrm{KZ}} )$.

The following is a key lemma 
which is a direct consequence of the result  shown in \cite{LM96b}.

\begin{lem}\label{key lemma}
Let $(\mu,\varphi)$ be any associator. 
Put $\gamma_{(\mu,\varphi)}$ to be the element
\footnote{
So we have $\gamma_{\mathrm{KZ}}=\gamma_{(2\pi\sqrt{-1},\varPhi_{\mathrm{KZ}})}$.}
 in ${\mathcal A}(S^1)$
replacing 
$\varPhi_{\mathrm{KZ}}(\frac{X_0}{2\pi\sqrt{-1}},\frac{X_1}{2\pi\sqrt{-1}})$
with
$\varphi(\frac{X_0}{\mu},\frac{X_1}{\mu})$
in Figure \ref{gamma}.
%$\zeta({\bf k})$ with $\zeta_{\varphi}({\bf k})$ and
%$2\pi\sqrt{-1}$ with $\mu$ in \eqref{gamma}.
Then $$\bar\gamma_{(\mu,\varphi)}=\bar\gamma_{\mathrm{KZ}}$$ holds in ${\mathcal A}_0(S^1)$.
%for any associator $(\mu,\varphi)$.
\end{lem}

\begin{proof}
We may assume $\mu=1$ because
$\gamma_{(\mu,\varphi)}=\gamma_{(1,\varphi({X_0}/{\mu},{X_1}/{\mu}))}$ by definition.
As is explained in \cite{LM96b},
our $\gamma_{(1,\varphi)}$  is nothing but the framed link {\it pre-}invariant
\footnote{
It is not an invariant of framed links but it will be so after one multiply it  with a suitable powers of $\gamma_{(1,\varphi)}$ (cf. \cite{LM96b} Thoerem 2.1)
}
$Z^F_f(U)$ 
of the so-called U-knot $U$ depicted in Figure \ref{U-knot}.
%%%%%%%%%%%%%%%%%%%%%%%%%%%%%%%%%%%%%%%%%%%%%%%%%%%%%%%
\begin{figure}[h]
\begin{center}
\begin{tikzpicture}[>=stealth,thick,xscale=0.3,yscale=0.6]
\foreach \x in {-3,-1,1,3}
\draw (\x,0)--(\x,1);
\draw[<-](-1,0) .. controls (-1,-1) and (1,-1) .. (1,0);
\draw[->](-3,0) .. controls (-3,-2.5) and (3,-2.5) .. (3,0);
\draw[<-](1,1) .. controls (1,2) and (3,2) .. (3,1);
\draw[<-](-3,1) .. controls (-3,2) and (-1,2) .. (-1,1);
\end{tikzpicture}
\caption{$U$}
\label{U-knot}
\end{center}
\end{figure}
%%%%%%%%%%%%%%%%%%%%%%%%%%%%%%%%%%%%%%%%%%%%%%%%%%%%%%%
Here $F$ is a \lq twistor' from 
$\varPhi_{\mathrm{KZ}}(\frac{X_0}{2\pi\sqrt{-1}},\frac{X_1}{2\pi\sqrt{-1}})$ 
to $\varphi$, that is,
a certain element $F$ in  ${\mathcal A}(\uparrow\uparrow)$ %where $X$ is two vertical lines,
such that 
$$
\varphi=F^{2,3}\cdot F^{1,23}\cdot\varPhi_{\mathrm{KZ}}(\frac{X_0}{2\pi\sqrt{-1}},\frac{X_1}{2\pi\sqrt{-1}})
\cdot (F^{12,3})^{-1}\cdot (F^{1,2})^{-1}
$$
holds in ${\mathcal A}(\uparrow\uparrow\uparrow)$ (for precise, see \cite{LM96b} \S 7).
According to  \cite{LM96b} Theorem 8 and its corollary,
%the framed knot (pre-)invariant
$Z^F_f(K)\in{\mathcal A}(S^1)$ for each framed link $K$
is independent of any choice of a twistor $F$,
namely, any choice of an associator $(1,\varphi)$.
\footnote{
Hence hereafter we will denote it simply by $Z_f(K)$.}
Thus
$$\gamma_{(\mu,\varphi)}=\gamma_{\mathrm{KZ}}$$
in ${\mathcal A}(S^1)$,
from which we obtain the claim.
\end{proof}

By Lemma \ref{key lemma}, we may replace $\bar\gamma_{\mathrm{KZ}}$ with $\bar\gamma_{(\mu,\varphi)}$
in \eqref{eqn:WN-LM95}. 
Therefore Le-Murakami relation (1995)  holds for any associator $(\mu,\varphi)$
because the equality \eqref{eq:gamma} replacing
$\zeta(k_1,\dots,k_m)$ with $\zeta_{\varphi}(k_1,\dots,k_m)$ and
$2\pi\sqrt{-1}$ with $\mu$ holds.
\qed
\medskip

(B) {\it Le-Murakami relation (1996):}
They constructed  a weight system (denoted $W$ in \cite{LM96} \S 3.2),
a linear map
$$W^{96}_N:{\mathcal A}(S^1)\to\mathbb{C}[[h]]$$
by using the fundamental representation of  the Lie algebra $\frak{so}(N)$
 for each $N\geqslant 3$.
In \cite{LM96} Section 3, they introduced an invariant 
%$\kappa(L)\in\mathbb{C} [[h]]$
of {\it un-}oriented framed links %$L$  
which is given by
\begin{equation}\label{kappa=K}
\kappa(L):=\frac{N^{s_L-2}W^{96}_N(Z_f(L))}{W^{96}_N(\gamma_{\mathrm{KZ}})^{s_L-1}}
\in\mathbb{C} [[h]]
\end{equation}
for framed link diagrams $L$ with $s_L$ maximal points.
They showed that 
that it satisfies  the followings:
\begin{equation*}
\kappa(L_r)=e^{(N-1)h}\kappa(L), \qquad
\kappa(L_l)=e^{-(N-1)h}\kappa(L),
\end{equation*}
\begin{equation*}
\kappa(L_+)-\kappa(L_-)=\{\exp(h)-\exp(-h)\}\cdot\{\kappa(L_0)-\kappa(L_\infty)\},
\end{equation*}
$$
\kappa(\unorientedcircle)=1.
$$
Here  $L_r$ (resp. $L_l$) is the same framed link diagram to $L$ 
with a right-handed (resp. left-handed) curl added (using a type I Reidemeister move) 
and $L_+$, $L_-$ , $L_0$, $L_\infty$ are all identical except within a ball
in which they differ as illustrated below
%$$
%\underset{L_+}{\huge{\nondirecteddiaCrossP}} 
%\qquad\quad \underset{L_-}{\huge{\nondirecteddiaCrossN}}
%\qquad\quad\underset{L_0}{\huge{\nondirecteddiaSmooth}}
%\qquad\quad\underset{L_\infty}{\huge{\nondirecteddiainfinity}}
%$$
%%%%%%%%%%%%%%%%%%%%%%%%%%%%%%%%%%%%%%%%%%%%%%%%%%%%%%%
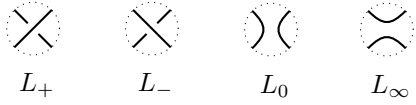
\begin{center}
%% non-dirceted positive crossing
    \begin{tikzpicture}(15,15)
       \draw[-,thick] (0.5,0)--(0,0.5) ;
       \draw[line width=0.2cm, white](0,0)--(0.5,0.5) ;
       \draw[-,thick] (0,0)--(0.5,0.5) ;
       \draw[dotted] (0.25,0.25) circle (0.35); 
       \draw(0.3,-0.5) node{$L_{+}$};
    \end{tikzpicture}
  \qquad
%% non-dirceted negative crossing
    \begin{tikzpicture}(15,15)
       \draw[-,thick] (0,0)--(0.5,0.5) ;
       \draw[line width=0.2cm, white](0.5,0)--(0,0.5) ;
       \draw[-,thick] (0.5,0)--(0,0.5) ;
       \draw[dotted] (0.25,0.25) circle (0.35);
       \draw(0.3,-0.5) node{$L_{-}$};
    \end{tikzpicture}
  \qquad
%%  non-dirceted smoothing
    \begin{tikzpicture}(15,15)
         \draw[-,thick](0,0) .. controls (0.2,0.2) and (0.2,0.3)  .. (0,0.5);
         \draw[-,thick](0.5,0) .. controls (0.3,0.2) and (0.3,0.3)  .. (0.5,0.5);
         \draw[dotted] (0.25,0.25) circle (0.35);
          \draw(0.3,-0.5) node{$L_{0}$};
    \end{tikzpicture}
  \qquad
%%  non-dirceted smoothing (infinitiy)
    \begin{tikzpicture}(15,15)
        \draw[-,thick](0,0) .. controls (0.2,0.2) and (0.3,0.2)  .. (0.5,0);
         \draw[-,thick](0,0.5) .. controls (0.2,0.3) and (0.3,0.3)  .. (0.5,0.5);
         \draw[dotted] (0.25,0.25) circle (0.35);
          \draw(0.3,-0.5) node{$L_{\infty}$};
    \end{tikzpicture}
\end{center}
%%%%%%%%%%%%%%%%%%%%%%%%%%%%%%%%%%%%%%%%%%%%%%%%%%%%%%%
and $\unorientedcircle$ means the trivial unoriented knot.
For the $2$-components trivial unoriented  link $L=\unorientedcircle\,\unorientedcircle$,
the above equations yield
$$
\kappa(\unorientedcircle\,\unorientedcircle)=\frac{e^{(N-1)h}-e^{-(N-1)h}}{e^h-e^{-h}}+1
$$
While the left hand side of \eqref{kappa=K} for
$L=\unorientedcircle\,\unorientedcircle$ was calculated to be
$$
\kappa(\unorientedcircle\,\unorientedcircle)=
\frac{W^{96}_N(\unorientedcircle\,\unorientedcircle)}{W^{96}_N(\gamma_{\mathrm{KZ}})}
=\frac{N^2}{W^{96}_N(\gamma_{\mathrm{KZ}})}
$$
by %$\hat{Z}(\orientedcircle)=\bar\gamma_{\mathrm{KZ}}^{-1}$ and
\cite{LM96} Section 3.2.
%By \eqref{W=P} and \eqref{homfly for 2-components tirvial link}, 
Thus the following  is obtained 
\begin{equation}\label{eqn:WN-LM96}
W^{96}_N(\gamma_{\mathrm{KZ}})=\frac{N^2\sinh h}{\sinh(N-1)h+\sinh h}.
\end{equation}
%by relating $W_N$ with the Kauffman knot polynomial (cf. Theorem 3.7).
In \cite{LM96} Section 4,
they constructed an algebra homomorphism
$$\psi:{\mathcal A}_0(S^1)\to{\mathcal A}(S^1)$$
such that 
\begin{equation}
\psi(\bar\gamma_\mathrm{KZ})=\gamma_\mathrm{KZ}.
\end{equation}
and in \cite{LM96} Proposition 4.5 showed that  
\begin{equation}
W^{96}_N\circ \psi(D_{\bf k})=(-2)^{{\rm wt}({\bf k})}Ng({\bf k})h^{{\rm wt}({\bf k})}.
\end{equation}
The above equations with \eqref{eq:gamma} immediately yields
the Le-Murakami relation (1996) for 
$(2\pi\sqrt{-1},\varPhi_{\mathrm{KZ}} )$.

Again by Lemma \ref{key lemma}, we may replace $\bar\gamma_{\mathrm{KZ}}$ with $\bar\gamma_{(\mu,\varphi)}$
in \eqref{eqn:WN-LM96}. 
Thus the validity of the relation for any associator $(\mu,\varphi)$ follows.
\qed
\medskip

(C) {\it Takamuki relation:}
Let $\varphi$ be any invertible series in 
$\mathbb C\langle\langle X_0,X_1\rangle\rangle$.
Consider the chord diagram $\delta_\varphi\in{\mathcal A}(S^1)$
defined as in Figure \ref{delta}.
%%%%%%%%%%%%%%%%%%%%%%%%%%%%%%%%%%%%%%%%%%%%%%%%%%%%%%%
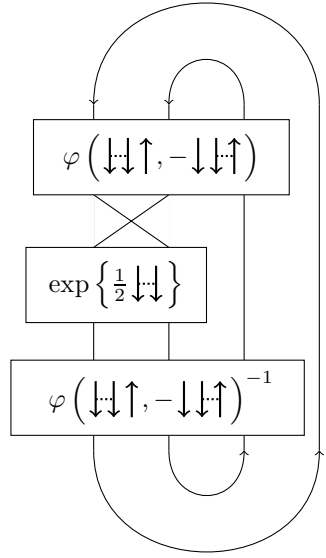
\begin{figure}[h]
\begin{center}
\begin{tikzpicture}[scale=1]
\foreach \x in {0,1,2,3}
\draw (\x,0.3)--(\x,4.9);

\draw[thin,fill=white] (-1.1,0.5) rectangle (2.8,1.5);
\draw (0.9,1) node{$\displaystyle \varphi \left(\raisebox{-1ex}{
\begin{tikzpicture}[thick,scale=0.8]
\draw[<-] (0,0)--(0,0.6);
\draw[densely dotted](0,0.3)--(0.3,0.3);
\draw[<-] (0.3,0)--(0.3,0.6);
\draw[->] (0.6,0) --(0.6,0.6);
\end{tikzpicture}
},-
\raisebox{-1ex}{
\begin{tikzpicture}[thick,scale=0.8]
\draw[<-] (-0.3,0) --(-0.3,0.6);
\draw[<-] (0,0)--(0,0.6);
\draw[densely dotted](0,0.3)--(0.3,0.3);
\draw[->] (0.3,0)--(0.3,0.6);
\end{tikzpicture}
}\right)^{-1}$};

\draw[thick,color=white] (0,3)--(0,3.7) (1,3)--(1,3.7) ;

\draw[thin,fill=white] (-0.9,2) rectangle (1.5,3);
\draw (0.3,2.5) node{$\exp\left\{\frac{1}{2}
\raisebox{-1ex}{
\begin{tikzpicture}[thick,scale=0.8]
\draw[<-] (0,0)--(0,0.6);
\draw[densely dotted](0,0.3)--(0.3,0.3);
\draw[<-] (0.3,0)--(0.3,0.6);
\end{tikzpicture}
}
\right\}$};

\draw (0,3.0)--(1,3.7) (1,3.0)--(0,3.7);

\draw[thin,fill=white] (-0.8,3.7) rectangle (2.6,4.7);
\draw (0.9,4.2) node{$\displaystyle \varphi \left(\raisebox{-1ex}{
\begin{tikzpicture}[thick,scale=0.8]
\draw[<-] (0,0)--(0,0.6);
\draw[densely dotted](0,0.3)--(0.3,0.3);
\draw[<-] (0.3,0)--(0.3,0.6);
\draw[->] (0.6,0) --(0.6,0.6);
\end{tikzpicture}
},-
\raisebox{-1ex}{
\begin{tikzpicture}[thick,scale=0.8]
\draw[<-] (-0.3,0) --(-0.3,0.6);
\draw[<-] (0,0)--(0,0.6);
\draw[densely dotted](0,0.3)--(0.3,0.3);
\draw[->] (0.3,0)--(0.3,0.6);
\end{tikzpicture}
}\right)$};

\draw[->](1,0.3) .. controls (1,-0.5) and (2,-0.5) .. (2,0.3);
\draw[->](0,0.3) .. controls (0,-1.5) and (3,-1.5) .. (3,0.3);

\draw[<-](1,4.9) .. controls (1,5.7) and (2,5.7) .. (2,4.9);
\draw[<-](0,4.9) .. controls (0,6.7) and (3,6.7) .. (3,4.9);
\end{tikzpicture}
\caption{$\delta_\varphi$}
\label{delta}
\end{center}
\end{figure}
%%%%%%%%%%%%%%%%%%%%%%%%%%%%%%%%%%%%%%%%%%%%%%%%%%%%%%%

Assume that $\varphi$  satisfies \eqref{shuffle}
and \eqref{2-cycle}.
Then by Lemma \ref{Le-Murakami method},
%and \ref{lem:duality},$\varphi$ and $\varphi^{-1}$ can be described as in \cite{T} \S 5:
\begin{align}\label{eqn:1}
 \varphi= 1+\sum_{k=1}^\infty
 \sum_{{\bf p},{\bf q},{\bf r},{\bf s}\in{\mathbb Z}_{\geqslant 0}^k}&
(-1)^{{\rm wt}({\bf p})+{\rm wt}({\bf s})}
\zeta_\varphi\bigl(\tau({\bf p}+{\bf r},{\bf q}+{\bf s})\bigr)
\binom{{\bf p}+{\bf r}}{{\bf r}}\binom{{\bf q}+{\bf s}}{{\bf s}}\\ \notag
& \cdot X_1^{{\rm wt}({\bf r})}
X_0^{q_k}X_1^{p_k}\cdots X_0^{q_1}X_1^{p_1}
X_0^{{\rm wt}({\bf s})} .
\end{align}
By Lemma \ref{lem:duality},
\begin{align}\label{eqn:2}
 \varphi= 1+\sum_{k=1}^\infty
 \sum_{{\bf p},{\bf q},{\bf r},{\bf s}\in{\mathbb Z}_{\geqslant 0}^k}&
(-1)^{{\rm wt}({\bf q})+{\rm wt}({\bf r})}
\zeta_\varphi\bigl(\tau({\bf p}+{\bf r},{\bf q}+{\bf s})\bigr)
\binom{{\bf p}+{\bf r}}{{\bf r}}\binom{{\bf q}+{\bf s}}{{\bf s}}\\ \notag
& \cdot X_1^{{\rm wt}({\bf s})}
X_0^{p_1}X_1^{q_1}\cdots X_0^{p_k}X_1^{q_k}
X_0^{{\rm wt}({\bf r})} .
\end{align}
By Lemma \ref{lem:duality} and \eqref{2-cycle},
\begin{align*}
 \varphi^{-1}= 1+\sum_{k=1}^\infty
 \sum_{{\bf p},{\bf q},{\bf r},{\bf s}\in{\mathbb Z}_{\geqslant 0}^k}&
 (-1)^{{\rm wt}({\bf q})+{\rm wt}({\bf r})}
 \zeta_\varphi\bigl(\tau({\bf p}+{\bf r},{\bf q}+{\bf s})\bigr)
\binom{{\bf p}+{\bf r}}{{\bf r}}\binom{{\bf q}+{\bf s}}{{\bf s}}\\
& \cdot X_0^{{\rm wt}({\bf s})}
X_1^{p_1}X_0^{q_1}\cdots X_1^{p_k}X_0^{q_k}
X_1^{{\rm wt}({\bf r})} .
\end{align*}
They are equal to the formulae in \cite{T} \S 5
when $\varphi(A,B)=\varPhi_{\mathrm{KZ}}(\frac{A}{2\pi\sqrt{-1}},\frac{B}{2\pi \sqrt{-1}})$.

By the equations \eqref{eqn:1} and \eqref{eqn:2},
the following can be deduced
%His computations in 
(cf. \cite{T} Proof of Proposition 5 and 6)
%are read as
\begin{equation}\label{eqn:WN-T}
W_N^{95}(\bar\delta_\varphi)=\frac{e^{-Nh}}{2}
\left\{
(e^h+e^{-h})%W_N^{95}(\tilde\delta_\varphi)
R_\varphi+N^2(e^h-e^{-h})
\right\}
\end{equation}
where $R_\varphi$ means the right hand side of Takamuki relation
replacing $\zeta({\bf k})$ with $\zeta_{\varphi}({\bf k})$
and $2\pi\sqrt{-1}$ with $1$.

\begin{lem}\label{}
Let $(\mu,\varphi)$ be any associator with $\mu=1$. Then
%$\delta_\varphi=\gamma_{(\mu,\varphi)}$ in ${\mathcal A}(S^1)$, hence 
$$\bar\delta_\varphi=\bar\gamma_{(\mu,\varphi)}$$ in ${\mathcal A}_0(S^1)$.
\end{lem}

\begin{proof}
%As is essentially explained in \cite{T},
Again as is explained in \cite{T},
the quotient $\bar\delta_\varphi\in{\mathcal A}_0(S^1)$ is nothing but 
Kontsevich's (\cite{Ko})  knot (pre-)invariant 
$Z^F(K_0)$ in \cite{LM96b}
of the unknot $K_0$ represented in Figure \ref{X-unknot}
with a twistor $F$ from $\varPhi_{\mathrm{KZ}}(\frac{X_0}{2\pi\sqrt{-1}},\frac{X_1}{2\pi\sqrt{-1}})$
to $\varphi$.
%%%%%%%%%%%%%%%%%%%%%%%%%%%%%%%%%%%%%%%%%%%%%%%%%%%%%%%
\begin{figure}[h]
\begin{center}
\begin{tikzpicture}[>=stealth,thick, rotate=90, xscale=0.6, yscale=0.4]
\draw[<-](-0.5,0) .. controls (-0.5,1) and (1.5,2) .. (1.8,0);
\draw[line width=0.2cm, white](-1.8,0) .. controls (-1.5,2) and (0.5,1) .. (0.5,0) .. controls (0.5,-1) and (-0.5,-1) .. (-0.5,0);
\draw[<-](-1.8,0) .. controls (-1.5,2) and (0.5,1) .. (0.5,0) .. controls (0.5,-1) and (-0.5,-1) .. (-0.5,0);
\draw[<-](1.8,0) .. controls (1.5,-3) and (-1.5, -3) .. (-1.8,0);
\end{tikzpicture}
\caption{$K_0$}
\label{X-unknot}
\end{center}
\end{figure}
%%%%%%%%%%%%%%%%%%%%%%%%%%%%%%%%%%%%%%%%%%%%%%%%%%%%%%%
Because $K_0$ has two maximums and the unknot diagram $\orientedcircle$ has a maximum
with $Z^F(\orientedcircle)=\bar\gamma_{(\mu,\varphi)}$
and they give equivalent knots,
we have $\bar\delta_\varphi\cdot (\bar\gamma_{(\mu,\varphi)})^{-2}=(\bar\gamma_{(\mu,\varphi)})^{-1}$.
\end{proof}

%Let us further assume that $(1,\varphi)$ forms an associator.
%Then  by Lemma \ref{}, $$\delta_\varphi=\gamma_\varphi.$$
Therefore
by combining  \eqref{eqn:WN-T} with \eqref{eqn:WN-LM95},
we obtain Takamuki relation for any associator $(1,\varphi)$
by the above lemma.
Hence the relation also holds for any associator $(\mu,\varphi)$
because $(1,\varphi(X_0/\mu,X_1/\mu))$ forms an associator.
\qed
\medskip

(D) {\it  Ihara-Takamuki relation:}
They constructed  a normalized weight system (denoted $\widehat{W_0}$ in \cite{IT} \S 3),
a linear map
$$W_0^{\mathrm{IT}}:{\mathcal A}_0(S^1)\to\mathbb{C}[[h]]$$
by using the $7$-dimensional irreducible standard representation $\Gamma_{1,0}$ of the exceptional 
simple Lie algebra of type $G_2$.
In  \cite{IT} Proposition 4, they showed 
\begin{equation}
W_0^{\mathrm{IT}}(D_{\tau({\bf p},{\bf q})})=w({\bf p},{\bf q})h^{{\rm wt}({\bf p})+{\rm wt}({\bf q})}
\end{equation}
and
\begin{equation}
w({\bf p},{\bf q})=w({\bf q^*},{\bf p^*}).
\end{equation} 
%where ${\bf r^*}:=(r_k,\dots,r_1)$ for ${\bf r}=(r_1,\dots,r_k)$.

From \cite{IT} Proposition 2, 3 and 5
where they related $W_0^{\mathrm{IT}}(\bar\gamma_{\mathrm{KZ}})$
with Kuperberg's computation \cite{Ku} on
the quantum dimension of the associated quantum representation,
it follows 
\begin{equation}\label{eqn:WN-IT}
W_0^{\mathrm{IT}}(\bar\gamma_{\mathrm{KZ}})=\frac{7^2\cdot [6]_q[4]_q}{[12]_q[7]_q[2]_q}.
\end{equation}
Then by combining the above three relation with 
the duality relation (cf. Lemma \ref{lem:duality}) among MZV's
\begin{equation}
\zeta\bigl(\tau({\bf p},{\bf q})\bigr)=
\zeta\bigl(\tau({\bf q^*},{\bf p^*})\bigr),
\end{equation} 
we obtain Ihara-Takamuki relation for
$(2\pi\sqrt{-1},\varPhi_{\mathrm \mathrm{KZ}} )$
since \eqref{eq:gamma} can be reformulated as
$$
\bar\gamma_{\mathrm{KZ}}=\orientedcircle+\sum_{\substack{{\bf p},{\bf q} \\  {\rm dp}({\bf p})={\rm dp}({\bf q})}}
\frac{(-1)^{{\rm wt}({\bf q})}}{(2\pi\sqrt{-1})^{{\rm wt}({\bf p})+{\rm wt}({\bf q})}}
\zeta\bigl(\tau({\bf p},{\bf q})\bigr)
D_{\tau({\bf p},{\bf q})}.
$$

Suppose that $(\mu,\varphi)$ be any associator.
Again by Lemma \ref{key lemma}, we may replace $\bar\gamma_{\mathrm{KZ}}$ with $\bar\gamma_{(\mu,\varphi)}$
in \eqref{eqn:WN-IT}. 
The validity of the duality relation for $(\mu,\varphi)$ is verified
in Lemma \ref{duality},
hence from which we obtain Ihara-Takamuki relation for $(\mu,\varphi)$.
\qed

\begin{rem}
There is another proof of the derivation of 
Le-Murakami relation (A) from the associator relation:
The relation (A)
is generalized to Ohno-Zagier relation (\cite{OZ}),
but which was shown in \cite{Li} to be a consequence of
the regularized double shuffle relation.
Whereas in \cite{F11} it was shown that the last relation 
can be derived from the associator relation.
\end{rem}

Though we do not have a precise definition of relations among MZV's in knot theory,
the author expects that all  such relations should be derived from
the associator relations
i.e. the pentagon equation and the shuffle relation.

\thanks{
{\it Acknowledgements}.
This work was supported by 
Grant-in-Aid for Young Scientists (A) 24684001,
JSPS KAKENHI JP15KK0159
and Daiko Foundation.
The author would like to thank Hisatoshi Kodani for helping him to
draw several pictures in the paper.
}

%%%%%%%%%%%%%%%%%%%%%%%%%%%%%%%%%%%%%%%%%%%%%%%%%%%%%%%%%%%%%%%%%%%%%%

\end{document}